\documentclass[11pt, reqno]{amsart}
\usepackage{amssymb}
\usepackage{amsmath}
\usepackage{amscd}
\usepackage{amssymb}
\usepackage{eucal}
\usepackage{amsmath}
\usepackage{amscd}
\usepackage[dvips]{color}
\usepackage{multicol}
\usepackage[all]{xy}           %xypic macro for latex2.09
\usepackage{graphicx}
\usepackage{color}
\usepackage{colordvi}
\usepackage{xspace}
\usepackage{tikz}

\usepackage{ifpdf}
\ifpdf
 \usepackage[colorlinks,final,backref=page,hyperindex]{hyperref}
\else
 \usepackage[colorlinks,final,backref=page,hyperindex,hypertex]{hyperref}
\fi

\newcommand{\delete}[1]{}

\allowdisplaybreaks

\newtheorem{theorem}{Theorem}[section]
\newtheorem{prop}[theorem]{Proposition}
\newtheorem{defn}[theorem]{Definition}
\newtheorem{lem}[theorem]{Lemma}
\newtheorem{coro}[theorem]{Corollary}

\newtheorem{thm}[theorem]{Theorem}

\newtheorem{rem}[theorem]{Remark}
\newtheorem{exam}[theorem]{Example}
\newcommand{\pair}[1]{\ensuremath{\left\langle #1 \right\rangle}}

\DeclareMathOperator{\ad}{ad}

\DeclareMathOperator{\Hom}{Hom}

\DeclareMathOperator{\End}{End}

\DeclareMathOperator{\id}{id}
\newcommand{\C}{\mathbb{C}}

\newcommand{\F}{\mathbb{F}}
\newcommand{\g}{\mathfrak{g}}

\newcommand{\gl}{\mathfrak{gl}}
\newcommand{\ssl}{\mathfrak{sl}}
\newcommand{\re}{\mathfrak{r}}
\newcommand{\LA}{\Longleftarrow}
\newcommand{\RA}{\Longrightarrow}

\newcommand{\ra}{\longrightarrow}
\newcommand{\A}{\mathcal{A}}
\newcommand{\B}{\mathcal{B}}
\newcommand{\Sol}{\mathcal{Sol}}

\newcommand{\OO}{\mathcal{O}}

\newcommand{\wt}{\widetilde}

\newcommand{\hbo}{$\hfill\Diamond$}

\begin{document}
\title{Skew-symmetric solutions of the classical Yang-Baxter equation and $\OO$-operators of Malcev algebras}
\def\shorttitle{Skew-symmetric solutions of the CYBE and $\OO$-operators of Malcev algebras}

\author{Shan Ren}
\address{School of Mathematics and Statistics, Northeast Normal University, Changchun, China}
\email{rens734@nenu.edu.cn}

\author{Runxuan Zhang}
\address{Department of Mathematical and Physical Sciences, Concordia University of Edmonton, Edmonton, Canada; School of Mathematics and Statistics, Northeast Normal University, Changchun, China}
\email{runxuan.zhang@concordia.ab.ca; zhangrx728@nenu.edu.cn}


\begin{abstract}
We study connections between skew-symmetric solutions of the classical Yang-Baxter equation (CYBE)   and  $\OO$-operators of Malcev algebras. We prove that a skew-symmetric solution of the CYBE on a Malcev algebra can be interpreted as an $\OO$-operator associated to the coadjoint representation. We show that this connection can be enhanced with symplectic forms when considering non-degenerate  skew-symmetric solutions. We also show that $\OO$-operators associated to a general representation could give skew-symmetric solutions of the CYBE on certain semi-direct product of Malcev algebras. We  reveal the relationship between invertible $\OO$-operators and compatible pre-Malcev algebra structures on a Malcev algebra. We finally obtain several analogous results on connections between  the CYBE and $\OO$-operators in the case of pre-Malcev algebras.
\end{abstract}%{\scriptsize {\scriptsize }}

%\date{\today}
%\subjclass[2010]{16T25; 17D10.}
\maketitle \baselineskip=15pt
\baselineskip=18pt
%%%%%%%%%%%%%%%%%%%%%%%%%%%Contents%%%%%%%%%%%%%%%%%%%%%%%%
%\textcolor{blue}{\tableofcontents{}}
%\dottedcontents{section}[1.16cm]{}{1.8em}{5pt}
%\dottedcontents{subsection}[2.00cm]{}{2.7em}{5pt}
%\dottedcontents{subsubsection}[2.86cm]{}{3.4em}{5pt}

%\vspace{-4mm}
%%%%%%%%%%%%%%%%%%%%%%%%%%%Sections%%%%%%%%%%%%%%%%%%%%%%%%
\section{Introduction}\label{sec1}
\setcounter{equation}{0}
\renewcommand{\theequation}
{1.\arabic{equation}}
\setcounter{theorem}{0}
\renewcommand{\thetheorem}
{1.\arabic{theorem}}

\noindent The classical Yang-Baxter equation (CYBE) on a finite-dimensional nonassociative algebra  of characteristic zero occupies a central place in connecting mathematics and mathematical physics.
The study of the CYBE on a Lie algebra $\g$ has substantial ramifications and applications in the areas of symplectic geometry, quantum groups,  integrable systems, and quantum field theory, whereas characterizing specific solutions  of the CYBE for a given $\g$ is an indispensable and challenging task in terms of the viewpoint of pure mathematics; see for example \cite{BD82, Sto99}. As a natural generalization of Lie algebras, Malcev algebras have been studied extensively since Malcev's work in the 1950s (\cite{Mal55}). Our primary objective is to give a systematic study on skew-symmetric solutions of the CYBE on Malcev algebras, stemming from the point of view of Kupershmidt in \cite[Section 2]{Kup99} that regards solutions of the CYBE as $\OO$-operators.
Our approach exposes some interesting connections between the CYBE, $\OO$-operators, and pre-Malcev algebras.

Let $A$ be a Malcev algebra over a field $\F$ of characteristic zero and  $r=\sum_{i}x_{i}\otimes y_{i}\in A\otimes A$. The equation
\begin{equation}
r_{12}r_{13}+r_{13}r_{23}-r_{23}r_{12}=0
\end{equation} is called the  classical Yang-Baxter equation on $A$, where
$$
r_{12}r_{13}=\sum_{i,j}x_{i}x_{j}\otimes y_{i}\otimes y_j,
r_{13}r_{23}=\sum_{i,j}x_{i}\otimes x_{j}\otimes y_{i}y_{j},
r_{23}r_{12}=\sum_{i,j}x_{j}\otimes x_{i}y_{j}\otimes y_{i}.
$$
%In this paper, we focus on skew-symmetric solutions of the CYBE on Malcev algebras  because each such solution induces a structure of a Malcev bialgebra as in the case of Lie algebras. on $A$ with the comultiplication  $\Delta_r(x)=\sum_{i}(x_{i}x\otimes y_{i}-x_{i}\otimes xy_i)$ for all $x\in A$; see \cite[Lemma 6]{Gon12}.
Recall that for a vector space $V$, an element $r\in V\otimes V$ is called \textbf{skew-symmetric} if $\sigma(r)=-r$, where $\sigma$ denotes the twist map on $V\otimes V$.
Comparing with $\OO$-operators of Lie algebras and introducing the notion of $\OO$-operators of Malcev algebras, our first main theorem provides a sufficient and necessary condition for a skew-symmetric element $r\in A\otimes A$ being a solution of the
CYBE on $A$. To articulate this result, we write $\Sol(A)$ for the set of all solutions of the CYBE on $A$ and denote by $\OO_A(V,\rho)$ the set of all $\OO$-operators associated to the representation $\rho:A\ra\End(V)$.
For a finite-dimensional vector space $V$ over $\F$, $V^*$ refers to the dual space of $V$ and for $r\in V\otimes V$,
we define $T_r$ to be the linear map from $V^*$ to $V$ by
\begin{equation}
\label{1.1}
\pair{\xi,T_r(\eta)}=\pair{\xi\otimes \eta, r}
\end{equation}
 for all $\xi, \eta\in V^*$, where $\pair{-,-}:V^*\times V\ra\F$ denotes the natural pairing.

\begin{thm}\label{thm1}
Let $A$ be a finite-dimensional Malcev algebra over a field $\F$ of characteristic zero and $r$ be a skew-symmetric element in $A\otimes A$.  Then $r\in \Sol(A)$ if and only if $T_r\in \OO_A(A^*,\ad^*)$, where $(A^*,\ad^*)$ denotes the coadjoint representation of $A$.
\end{thm}

More significantly, specializing in non-degenerate skew-symmetric element $r\in A\otimes A$, we could
associate $r$ with a bilinear form $\B_r$ defined by $T_r^{-1}$ and the natural pairing.
Our second major result demonstrates that such $r$ is a solution of the CYBE on $A$ if and only if $\B_r$ is a
symplectic form on $A$. This result provides a possible way to explicitly describe all non-degenerate skew-symmetric solutions in $\Sol(A)$ for some specific Malcev algebras; see Example \ref{exam3.5}. To state this result, we recall that an element  $r\in A\otimes A$ is \textbf{non-degenerate} if $T_r$ defined by Eq. \eqref{1.1} is invertible; a bilinear form $\B$ on $A$ is \textbf{symplectic} if $\B(xy,z)+\B(yz,x)+\B(zx,y)=0$ for all $x,y,z\in A$.
Given a non-degenerate element $r\in A\otimes A$, we define the bilinear form $\B_r$ on $A$  by
$\B_r(x,y):=\pair{T_r^{-1}(x),y}$
for all $x,y\in A$. Lemma \ref{lem3.4} below shows that $r$ is skew-symmetric if and only if $\B_r$ is skew-symmetric.

\begin{thm}\label{thm2}
Let $A$ be a finite-dimensional Malcev algebra over a field $\F$ of characteristic zero  and $r\in A\otimes A$ be skew-symmetric and non-degenerate.  Then $r\in \Sol(A)$ if and only if $\B_r$ is a symplectic form on $A$.
\end{thm}

Our third  result provides a construction of a skew-symmetric solution of the CYBE on the semi-direct product Malcev algebra $A\ltimes_{\rho^*}V^*$ by an arbitrary $\OO$-operator associated to a given representation $(V,\rho)$ of a Malcev algebra $A$, which  reveals an inverse procedure of Theorem \ref{thm1} by loosing the restriction of coadjoint representations; compared with the case of Lie algebra (\cite[Section 2]{Bai07}).
Using the tensor-hom adjunction, we  identify $\Hom(V,A)$ with $A\otimes V^*$, and we identify an arbitrary element of
$A\otimes V^*$ with the image in $(A\oplus V^*)\otimes (A\oplus V^*)$ under the tensor product of the standard embeddings $A\ra A\oplus V^*$ and $V^*\ra A\oplus V^*$. Hence, given a linear map $T: V\ra A$, we  define an element $\wt{T}$ (see Eq. (\ref{eq3.1}) below) in $(A\oplus V^*)\otimes (A\oplus V^*)$ via this two identifications. Then we define a skew-symmetric element  $r_T:=\wt{T}-\sigma(\wt{T})$.

\begin{thm}\label{thm3}
Let $(V, \rho)$ be a  representation of a finite-dimensional Malcev algebra  $A$ over a field $\F$ of characteristic zero and $T: V\ra A$ be a  linear map. Then $r_T\in \Sol(A\ltimes_{\rho^*}V^*)$  if and only if $T\in\OO_A(V,\rho)$.
\end{thm}

% As an application, Example \ref{exam3.6} gives rise to a Malcev bialgebra structure on $A\ltimes_{\rho^*}V^*$ via a skew-symmetric solution $r_T$.

Moreover, invertible elements in $\OO_A(V,\rho)$ have a close relationship with compatible pre-Malcev algebra structures on $A$; see \cite[Section 2]{Mad17} for more details on pre-Malcev algebras. As pre-Lie algebras are Lie-admissible,  pre-Malcev algebras are   Malcev-admissible algebras in the sense of \cite{Myu86}.
Let $\A$ be a pre-Malcev algebra. Then the commutator $xy=x\cdot y-y\cdot x$ for all $x, y \in \A$ defines a Malcev algebra $[\A]$, which is called  the subadjacent Malcev algebra of $\A$, and we call $\A$ a compatible pre-Malcev algebra of $[\A]$. For an element $x\in \A$, the left multiplication operator $L_x: \A \ra \A$ sends  $y\in \A$ to $x\cdot y$. Then the linear map $L: [\A]\ra\End(\A)$ with $x\mapsto L_x$ gives a representation of the Malcev algebra $[\A]$.  Now our fourth theorem can be summarized as follows.

\begin{thm} \label{thm4}
Let $(V,\rho)$ be a representation of a finite-dimensional Malcev algebra  $A$ over a field $\F$ of characteristic zero. For an invertible element $T\in\OO_{A}(V,\rho)$, there
exists a compatible pre-Malcev algebra structure $\A_T$ on $A$ defined by
$x\cdot y:=T(\rho(x)T^{-1}(y))$
for all $x,y\in A$. Conversely, if there exists a compatible pre-Malcev algebra $\A$ on $A$, then
the identity map $\id_A$ belongs to $\OO_A(\A, L)$.
\end{thm}

 We also provide two applications of the existence of compatible pre-Malcev algebra structures on  Malcev algebras  to construct skew-symmetric solutions of the CYBE; see Corollary \ref{coro4.4}.
Our last several results concerned with  the CYBE on pre-Malcev algebras and $\OO$-operators can be regarded as an analogue of the theorems mentioned above. %Instead of skew-symmetric solutions in the case of Malcev algebras,
Compared with the case of pre-Lie algebras (\cite[Section 2]{Bai10}), these results focus on revealing links between symmetric solutions of the CYBE, $\OO$-operators and bilinear forms on pre-Malcev algebras; see Theorems \ref{thm4.6}, \ref{thm4.7} and Proposition \ref{prop4.8}.

\subsection*{Organization} In Section \ref{sec2}, we present some fundamental results on representations of Malcev algebras and $\OO$-operators, and then we develop two lemmas to prove Theorem \ref{thm1}. Section \ref{sec3} contains the proofs of
Theorems \ref{thm2} and \ref{thm3}, which are both closely related to Theorem \ref{thm1}. Theorem \ref{thm2} specializes in  the case
of Malcev algebras admitting  a symplectic form and Theorem \ref{thm3}  extends $\OO$-operators associated with the coadjoint representation $(A^*,\ad^*)$ to those associated with an arbitrary representation $(V,\rho)$.
In Section \ref{sec4}, we establish connections between invertible $\OO$-operators and compatible pre-Malcev structures on a Malcev algebra. We prove Theorem \ref{thm4} and produce several results about symmetric solutions of the CYBE on pre-Malcev algebras.

%\subsection*{Conventions}
Throughout this article we assume that the ground field $\F$ is a field of characteristic zero and all algebras, vector spaces and representations are finite-dimensional over $\F$. The multiplication in a Malcev algebra $A$ is denoted by $xy$ for all $x, y\in A$,  while  the multiplication in a pre-Malcev algebra $\A$ is denoted by $x\cdot y$ for all $x, y\in \A$.
 %Given a vector space $V$, we denote by $V^{\otimes k}$ the tensor product of $k$ copies of $V$ for $k\geqslant 3$ and denote by $\id_V$  the identity map on  $V$.

%\subsection*{Acknowledgements} We thank Yin Chen for his comments on the first version of the article. This research was partially supported by NNSF of China (No. 11301061).

\section{Malcev Algebras and $\OO$-operators}\label{sec2}
\setcounter{equation}{0}
\renewcommand{\theequation}
{2.\arabic{equation}}
\setcounter{theorem}{0}
\renewcommand{\thetheorem}
{2.\arabic{theorem}}

\noindent We recall some fundamental concepts  on representations of Malcev algebras. Comparing with
$\OO$-operators of Lie algebras, we introduce the notion of $\OO$-operators of Malcev algebras and present concrete
examples on  some specific Malcev algebras. We close this section by giving a proof of Theorem \ref{thm1}.

\subsection{Representations of Malcev algebras}

Recall that a nonassociative anti-commutative algebra $A$ over a field $\F$ is called a \textbf{Malcev algebra} provided that
\begin{equation}\tag{Malcev identity}
\label{MI}
(xy)(xz)=((xy)z)x+((yz)x)x+((zx)x)y
\end{equation}
for all $x, y, z \in A$. Compared with the relationship between Lie algebras and Lie groups, Malcev algebras
appear as the tangent spaces of smooth Moufang loops at the identities; see for example \cite{Mal55} for more backgrounds. It was proved in \cite[Proposition 2.21]{Sag61} that Malcev identity is also equivalent to
\begin{equation}\tag{Sagle identity}
\label{SI}
(xz)(yt) = ((xy)z)t + ((yz)t)x + ((zt)x)y + ((tx)y)z
\end{equation}
for all $x, y, z,t \in A$. Note that each Lie algebra is a Malcev algebra, thus all Lie-admissible algebras are Malcev-admissible. Here we have an example of a 4-dimensional non-Lie Malcev algebra.

\begin{exam}\label{exam2.1}
{\rm
Let $A$ be a vector space over $\F$ with a basis $\{e_{1}, e_{2}, e_{3}, e_{4}\}$. A direct calculation verifies that these non-zero products: $e_{1}e_{2}=-e_{2}, e_{1}e_{3}=-e_{3}, e_{1}e_{4}=e_{4}, e_{2}e_{3}=2e_{4}$, give rise to a non-Lie Malcev algebra structure on $A$; see \cite[Section 3]{Sag61}.
\hbo}\end{exam}

Let $A$ be a Malcev algebra over $\F$. A pair $(V,\rho)$ of  a vector space $V$ over $\F$ and a linear map $\rho:A\ra\gl(V)$ is called a \textbf{representation} of $A$ if
\begin{equation}\label{eq2.1}
\rho((xy)z)=\rho(x)\rho(y)\rho(z)-\rho(z)\rho(x)\rho(y)+\rho(y)\rho(zx)-\rho(yz)\rho(x)
\end{equation}
for all $x,y,z\in A$. Note that when $A$ is a Lie algebra, a Malcev representation of $A$  is not necessarily a Lie representation; see for example \cite[Section 3]{Yam63} and \cite{Eld90}. Two representations $(V_1, \rho_1)$ and $(V_2, \rho_2)$ are \textbf{isomorphic} if there exists a linear isomorphism $\varphi:V_2\ra V_1$ such that
$\rho_1(x)\circ \varphi=\varphi\circ\rho_2(x)$ for all $x\in A$.

Given a representation $(V, \rho)$ of $A$, there exists a Malcev algebra structure on
the direct sum $A\oplus V$ of  vector spaces given by
 \begin{equation}
 (x,u)(y,v)=(xy,\rho(x)v-\rho(y)u)
 \end{equation}
for all $x,y \in A$ and $u,v\in V$.  This Malcev algebra is called the
\textbf{semi-direct product} of $A$ and $V$ and  denoted by $A\ltimes_\rho V$. Moreover, consider the dual space $V^*$ of $V$ and a natural pairing $\pair{-,-}: V^*\times V\ra\F$. The \textbf{dual representation} $(V^*,\rho^*)$ of $(V,\rho)$
is defined by $\pair{\rho^*(x)\xi,v}=-\pair{\xi,\rho(x)v}$
for all $x\in A, \xi\in V^*$ and $v\in V$.
See  \cite{Kuz14} for a survey on structures and representations of Malcev algebras.
The following two examples of representations are necessary to us.

\begin{exam}\label{exam2.2}
{\rm Let $A$ be a Malcev algebra over $\F$.
As in the case of Lie algebras, the linear map $\ad:A\ra\End(A)$ sending $x$ to $\ad_x$, where $\ad_x(y)=xy$ for all $y\in A$, together with $A$, forms a representation $(A,\ad)$ of $A$, which is called the \textbf{adjoint representation} of $A$. The corresponding dual representation $(A^*,\ad^*)$ is called the \textbf{coadjoint representation} of $A$.
\hbo}\end{exam}

\subsection{$\OO$-operators of Malcev algebras}
Comparing with $\OO$-operators of Lie algebras \cite[Section 2]{Kup99}, we introduce
the notion of an $\OO$-operator of a Malcev algebra that also generalizes the concept of a Rota-Baxter operator (of weight zero) on a Malcev algebra appeared in \cite[Definition 8]{Mad17}.

\begin{defn}{\rm
Let $A$ be a Malcev algebra over $\F$ and $(V, \rho)$ be a representation of $A$.  A linear map $T: V \ra
A$ is called an \textbf{$\mathcal {O}$-operator} of
$A$ associated to  $(V,\rho)$ if
\begin{equation}
T(v)T(w) = T(\rho(T(v))w-\rho(T(w))v)
\end{equation}
for all $v, w \in V$. As stated previously, we write $\OO_A(V,\rho)$ for the set of all $\OO$-operators of $A$
associated to $(V, \rho)$. In particular, Rota-Baxter operators (of weight 0) of $A$ are nothing but $\OO$-operators
associated to $(A, \ad)$.
\hbo}\end{defn}

\begin{exam}\label{exam2.5}
{\rm Continued with Example \ref{exam2.1}, we consider the coadjoint representation $(A^*,\ad^*)$ of $A$.
Let $\{\varepsilon_1,\varepsilon_2,\varepsilon_3,\varepsilon_4\}$ be the basis of $A^*$ dual to $\{e_1,e_2,e_3,e_4\}$. With respect to  the two bases,
a linear map $T:A^*\ra A$ corresponds to a $4\times 4$-matrix. A direct verification shows that the following matrices
$$\begin{pmatrix}
   0   & 0&0&a   \\
   0   & 0&0&b  \\
    0 & 0&0&c  \\
  -a     & -b&-c&d   \\
\end{pmatrix}, \begin{pmatrix}
   0   & 0&0&0   \\
   0   & 0&0&a  \\
    0 & 0&0&b  \\
  c     & d&e&f   \\
\end{pmatrix}, \begin{pmatrix}
   0   & 0&0&a   \\
   0   & 2a^2/k&a&b  \\
    0 & 2a&k&c  \\
  -a     & -b&-c&d   \\
\end{pmatrix}$$
are $\OO$-operators of $A$ associated to $(A^*,\ad^*)$, where $a,b,c,d,e,f\in \F$ and $k\in\F \setminus\{0\}$.
\hbo}\end{exam}

\begin{exam}\label{exam2.6}
{\rm Consider the 3-dimensional simple Lie algebra $\ssl_2(\C)$ spanned by $\{x,y,z\}$ with nontrivial relations
$[x,y]=2y,[x,z]=-2z$ and $[y,z]=x$, which can also be viewed as a Malcev algebra. Suppose that $V$ is a vector space spanned by $\{u,v\}$.
It was proved in \cite[Section 6]{Car76} that the action of $\ssl_2(\C)$ on $V$ given by
$$xu=-2u, xv=2v,yu=0, yv=-2u,  zu=-2v, zv=0$$
makes $V$ become an irreducible non-Lie Malcev representation of $\ssl_2(\C)$. One can verify that
$$\begin{pmatrix}
     a & 2b & 0 \\
     b & c &0
\end{pmatrix}\textrm{ and  }\begin{pmatrix}
     a & 0 & 0 \\
     b & 0 &-2a
\end{pmatrix}$$
 are $\OO$-operators of $\ssl_2(\C)$ associated to this representation, where $a,b,c\in\C$.
\hbo}\end{exam}

\begin{lem}\label{lem2.6}
Suppose that $V$ is a vector space over $\F$ and $r=\sum_{i}x_{i}\otimes y_{i}\in V\otimes V$, $\xi\in V^*$. Then $T_r(\xi)=\sum_i\pair{\xi,y_i}x_i$. In particular, if $r$ is skew-symmetric, then $T_r(\xi)=-\sum_i\pair{\xi,x_i}y_i$.
\end{lem}

\begin{proof}
Given any $\eta\in V^*$, we have $\pair{\eta,T_r(\xi)}=\pair{\eta\otimes\xi,r}=\sum_{i}\pair{\eta\otimes\xi,x_{i}\otimes y_{i}}=
\sum_{i}\pair{\eta,x_{i}}\pair{\xi,y_{i}}=\pair{\eta,\sum_i\pair{\xi,y_i}x_i}$. Thus $\pair{\eta,T_r(\xi)-\sum_i\pair{\xi,y_i}x_i}=0$. As
the natural pairing is non-degenerate, it follows that $T_r(\xi)-\sum_i\pair{\xi,y_i}x_i=0$. For the second statement, recall that
$\sigma(r)=-r$ and we see that $\pair{\eta,T_r(\xi)}=\pair{\eta\otimes \xi,r}=-\pair{\eta\otimes \xi,\sigma(r)}=-\sum_{i}\pair{\eta\otimes \xi,y_{i}\otimes x_{i}}=-\pair{\eta,\sum_i\pair{\xi,x_i}y_i}$. The same reason as before implies that $T_r(\xi)=-\sum_i\pair{\xi,x_i}y_i$, as desired.
\end{proof}

\begin{lem}\label{lem2.7}
Let $V$ be a vector space over $\F$ and $r=\sum_{i}x_{i}\otimes y_{i}\in V\otimes V$. Then $r$ is skew-symmetric
if and only if $\pair{\xi,T_r(\eta)}=-\pair{\eta, T_r(\xi)}$ for all $\xi,\eta\in V^*$.
\end{lem}

\begin{proof} ($\RA$)
Since $r$ is skew-symmetric, we see that $-\pair{\eta, T_r(\xi)}=-\pair{\eta\otimes \xi,r}=\pair{\eta\otimes \xi,\sigma(r)}=
\sum_{i}\pair{\eta\otimes \xi,y_{i}\otimes x_{i}}=\pair{\xi,\sum_i\pair{\eta,y_i}x_i}=\pair{\xi,T_r(\eta)}$.
The last equation follows from Lemma \ref{lem2.6}.  ($\LA$) Note that
\begin{eqnarray*}
\pair{\eta\otimes \xi, \sigma(r)+r}&=&\pair{\eta\otimes \xi, \sigma(r)}+\pair{\eta\otimes \xi, r}\\
&=&\pair{\xi,T_r(\eta)}+\sum_i\pair{\eta,x_i}\pair{\xi,y_i}\quad (\textrm{by Lemma }\ref{lem2.6})\\
&=&-\pair{\eta, T_r(\xi)}+\sum_i\pair{\eta,x_i}\pair{\xi,y_i} \quad (\textrm{by the assumption})\\
&=&-\sum_i\pair{\xi,y_i}\pair{\eta,x_i}+\sum_i\pair{\eta,x_i}\pair{\xi,y_i}\\
&=&0.
\end{eqnarray*}
As the natural pairing is non-degenerate, $\sigma(r)+r=0$, i.e., $r$ is skew-symmetric.
\end{proof}

We are ready to prove our first theorem.

\begin{proof}[Proof of Theorem \ref{thm1}]
We first assume that $r=\sum_{i}x_{i}\otimes y_{i}\in A\otimes A$ is skew-symmetric.
Note that $(A^{\otimes 3})^*= (A^*)^{\otimes 3}$ when $A$ is finite dimensional. For arbitrarily chosen $\xi,\eta, \zeta\in A^*$,
 %this isomorphism allows us to identify  $\xi\otimes \eta\otimes\zeta$ with an element of $(A^{\otimes 3})^*$ and conversely, any element of $(A^{\otimes 3})^*$ is of this form.
 we consider the natural pairing on $A^{\otimes 3}$ and see that
\begin{eqnarray*}
\pair{\xi\otimes\eta\otimes\zeta, r_{12}r_{13}} & = &  \sum_{i,j}\pair{\xi\otimes\eta\otimes\zeta,x_{i}x_j\otimes y_{i}\otimes y_j}\\
 & = & \sum_{i,j}\pair{\xi,x_{i}x_j}\pair{\eta,y_{i}}\pair{\zeta,y_j}\\
  & = & \sum_{i,j}\pair{\xi,\left(\pair{\eta,y_{i}}x_{i}\right)\left(\pair{\zeta,y_j}x_j\right)}.
\end{eqnarray*}
On the other hand, it follows from Lemma \ref{lem2.6} that $\pair{\xi,T_r(\eta)T_r(\zeta)}= \sum_{i, j}\pair{\xi,\left(\pair{\eta,y_i}x_i\right)\left(\pair{\zeta,y_j}x_j\right)}$. Thus $\pair{\xi\otimes\eta\otimes\zeta, r_{12}r_{13}}=\pair{\xi,T_r(\eta)T_r(\zeta)}$.
Similarly, we have $\pair{\xi\otimes\eta\otimes\zeta, r_{13}r_{23}}=\pair{\zeta,T_r(\xi)T_r(\eta)}$ and
$\pair{\xi\otimes\eta\otimes\zeta, r_{23}r_{12}}=\pair{\eta,T_r(\xi)T_r(\zeta)}$. Hence,
\begin{equation}
\label{eq2.4}
\pair{\xi\otimes\eta\otimes\zeta, r_{12}r_{13}+r_{13}r_{23}-r_{23}r_{12}}=\pair{\xi,T_r(\eta)T_r(\zeta)}+\pair{\zeta,T_r(\xi)T_r(\eta)}-\pair{\eta,T_r(\xi)T_r(\zeta)}.
\end{equation}
This key equation involves the CYBE. To establish links between this equation and $\OO$-operators, we consider the
coadjoint representation $(A^*,\ad^*)$ of $A$ and note that
\begin{eqnarray*}
\pair{\xi, T_r(\ad^*_{T_r(\eta)}(\zeta))}&=& -\pair{\ad^*_{T_r(\eta)}(\zeta), T_r(\xi)}\quad \textrm{(by Lemma \ref{lem2.7})}\\
&=& \pair{\zeta, \ad_{T_r(\eta)}(T_r(\xi))}\\
&=&\pair{\zeta, T_r(\eta)T_r(\xi)}
= -\pair{\zeta, T_r(\xi)T_r(\eta)}.
\end{eqnarray*}
Similarly, $\pair{\xi, T_r(\ad^*_{T_r(\zeta)}(\eta))}= -\pair{\eta, T_r(\xi)T_r(\zeta)}.$ Hence,
\begin{eqnarray}\label{eq2.5}
&&\pair{\xi,T_r(\eta)T_r(\zeta)-T_r(\ad^*_{T_r(\eta)}(\zeta))+T_r(\ad^*_{T_r(\zeta)}(\eta))}\nonumber\\
&=&\pair{\xi, T_r(\eta)T_r(\zeta)}+\pair{\zeta, T_r(\xi)T_r(\eta)}-\pair{\eta, T_r(\xi)T_r(\zeta)}\\
&=& \pair{\xi\otimes \eta\otimes\zeta, r_{12}r_{13}+r_{13}r_{23}-r_{23}r_{12}}.\nonumber
\end{eqnarray}
Here the last equation follows from Eq. (\ref{eq2.4}).

Now we are in a position to complete the proof. In fact, if $r\in\Sol(A)$, then $r_{12}r_{13}+r_{13}r_{23}-r_{23}r_{12}=0$. Thus it follows from Eq. (\ref{eq2.5}) that $\pair{\xi,T_r(\eta)T_r(\zeta)-T_r(\ad^*_{T_r(\eta)}(\zeta))+T_r(\ad^*_{T_r(\zeta)}(\eta))}=0$. Since $\xi$ is arbitrary, we see that $T_r(\eta)T_r(\zeta)-T_r(\ad^*_{T_r(\eta)}(\zeta))+T_r(\ad^*_{T_r(\zeta)}(\eta))=0$ for all $\eta,\zeta\in A^*$, i.e.,
$T_r\in\OO_A(A^*,\ad^*)$. Conversely, assume that $T_r\in\OO_A(A^*,\ad^*)$. By Eq. (\ref{eq2.5}), we see that
$\pair{\xi\otimes \eta\otimes \zeta, r_{12}r_{13}+r_{13}r_{23}-r_{23}r_{12}}=0$. Therefore, we have
$r_{12}r_{13}+r_{13}r_{23}-r_{23}r_{12}=0$, showing that $r$ is a solution of the CYBE on $A$.
\end{proof}

\section{Bilinear Forms, the CYBE and Semi-direct Products}\label{sec3}
\setcounter{equation}{0}
\renewcommand{\theequation}
{3.\arabic{equation}}
\setcounter{theorem}{0}
\renewcommand{\thetheorem}
{3.\arabic{theorem}}

\noindent Specializing in Malcev algebras admitting non-degenerate invariant bilinear forms, we
establish an analogue of Theorem \ref{thm1} in which $\OO$-operators could be replaced by  Rota-Baxter operators of weight zero; see Corollary \ref{coro3.3}. We also give detailed proofs of Theorems \ref{thm2} and \ref{thm3}.
Throughout this section we let $A$ be a Malcev algebra over $\F$.

\subsection{Invariant bilinear forms}  A bilinear form $\B:A\times A\ra \F$ is called \textbf{invariant}  if  $\B(xy,z)=\B(x,yz)$  for all $x,y,z\in A$. %; see \cite[Section 5]{Gon12} for more details.

\begin{prop}\label{prop3.1}
Let $A$ be a Malcev algebra over $\F$. Then the adjoint representation $(A,\ad)$ and the coadjoint representation $(A^*,\ad^{*})$ of $A$ are isomorphic if and only if $A$ admits a non-degenerate invariant bilinear form.
\end{prop}

\begin{proof}
$(\RA)$ Suppose $\varphi:A\ra A^*$ is a linear isomorphism such that $\ad^*_x\circ \varphi=\varphi\circ\ad_x$ for arbitrary $x\in A$. Thus $\varphi(\ad_x(y))=\ad^*_x(\varphi(y))$ for all $y\in A$. We define a bilinear form
$\B_\varphi:A\times A\ra\F$ by $(x,y)\mapsto \pair{\varphi(x),y}$. To see that $\B_\varphi$ is invariant, we take $z\in A$, then $\B_\varphi(xy,z)=-\B_\varphi(yx,z)=-\B_\varphi(\ad_y(x),z)=-\pair{\varphi(\ad_y(x)),z}=
-\pair{\ad_y^*(\varphi(x)),z}=\pair{\varphi(x),\ad_y(z)}=\B_\varphi(x,yz)$, which means $\B_\varphi$ is invariant.
As $\varphi$ is bijective and the natural pairing on $A$ is non-degenerate, it follows that $\B_\varphi$ is also non-degenerate.

$(\LA)$ Assume that there exists a non-degenerate invariant bilinear form $\B$ on $A$. We define
\begin{equation}
\label{eq3.0}
\varphi_\B:A\ra A^*\textrm{ by }x\mapsto \B_x,
\end{equation}
where $\B_x(y):=\B(x,y)$ for all $y\in A$.
We first note that $\varphi_\B$
is linear as $\B$ is bilinear. To see that $\varphi_\B$ is bijective, assume that $x_1,x_2\in A$ are two elements such that
$\B_{x_1}=\B_{x_2}$. Then $\B(x_1,y)=\B(x_2,y)$ for all $y\in A$, i.e., $\B(x_1-x_2,y)=0$. As $\B$ is non-degenerate, we have $x_1=x_2$. Thus $\varphi_\B$ is injective. This fact, together with $\dim(A)=\dim(A^*)$, implies that $\varphi_\B$ is surjective. Hence, $\varphi_\B$ is a linear isomorphism. Moreover, we  choose a natural pairing on $A$ such that $\pair{\B_x,y}=\B(x,y)$ for all $x,y\in A$ as $\B$ is non-degenerate.  Since $\B$ is invariant, for all $x,y,z\in A$, we see that
$\pair{\varphi_\B(\ad_x(y))-\ad^*_x(\varphi_\B(y)),z}=\pair{\B_{xy},z}-\pair{\ad^*_x(\B_y),z}=
 \pair{\B_{xy},z}+\pair{\B_y,\ad_x(z)}= \pair{\B_{xy},z}+\pair{\B_y,xz}=\B(xy,z)+\B(y,xz)=-\B(yx,z)+\B(y,xz)=-\B(y,xz)+\B(y,xz)=0$. Hence, $\varphi_\B(\ad_x(y))=\ad^*_x(\varphi_\B(y))$ for all $y\in A$, i.e.,
$\varphi_\B\circ \ad_x=\ad^*_x\circ\varphi_\B$.
Therefore, $\varphi_\B$ is an isomorphism between $(A,\ad)$ and  $(A^*,\ad^{*})$.
\end{proof}

\begin{prop}\label{prop3.2}
Let $\varphi:(V_1,\rho_1)\ra(V_2,\rho_2)$ be an isomorphism of two representations of a Malcev algebra $A$ over $\F$.
Then for each  $T\in \OO_A(V_2,\rho_2)$,  the composition  $T \circ\varphi \in \OO_A(V_1,\rho_1)$. In particular, there is a one-to-one correspondence between $\OO_A(V_1,\rho_1)$ and $\OO_A(V_2,\rho_2)$ when the representations $(V_1,\rho_1)$ and $(V_2,\rho_2)$ are  isomorphic.
\end{prop}

\begin{proof}
Suppose $T\in \OO_A(V_2,\rho_2)$ and $v,w\in V_1$ are arbitrary elements. Since $\varphi$ is an isomorphism of representations, we have
\begin{eqnarray*}
&&(T\circ\varphi)(\rho_1((T\circ\varphi)(v))w-\rho_1((T\circ\varphi)(w))v)\\
%&=&T\circ\varphi(\rho_1(T\circ\varphi(v))\varphi^{-1}\varphi(w)-\rho_1(T\circ\varphi(w))\varphi^{-1}\varphi(v))\\
&=&T(\varphi(\rho_1(T(\varphi(v)))w)-\varphi (\rho_1(T(\varphi(w)))v))\\
&=&T(\rho_2(T(\varphi(v)))\varphi(w)-\rho_2(T(\varphi(w)))\varphi(v))\\
&=&T(\varphi(v))T(\varphi(w))=(T\circ\varphi)(v)(T\circ\varphi)(w),
\end{eqnarray*}
which implies that $T\circ\varphi\in \OO_A(V_1,\rho_1)$. Similarly, for each $S\in \OO_A(V_1,\rho_1)$, one can show that
$S\circ\varphi^{-1}\in \OO_A(V_2,\rho_2)$. We define a map $\Phi: \OO_A(V_1,\rho_1)\ra \OO_A(V_2,\rho_2)$ by
sending each $S$ to $S\circ\varphi^{-1}$ and another map $\Psi: \OO_A(V_2,\rho_2)\ra \OO_A(V_1,\rho_1)$ by sending every $T$ to $T\circ\varphi$. Clearly, $\Psi\circ\Phi=1_{\OO_A(V_1,\rho_1)}$ and $\Phi\circ\Psi=1_{\OO_A(V_2,\rho_2)}$.
Thus  $\Psi$ is a bijection and the proof is completed.
\end{proof}

Together with Theorem \ref{thm1}, Propositions \ref{prop3.1} and \ref{prop3.2}, imply the following result.

\begin{coro}\label{coro3.3}
Let $A$ be a Malcev algebra admitting a non-degenerate invariant bilinear form $\B$ and $r$ be a skew-symmetric element in $A\otimes A$.  Then $r\in \Sol(A)$ if and only if $T_r\circ\varphi_\B$ is a Rota-Baxter operator of weight zero on $A$, where $T_r$ and $\varphi_\B$ are defined as in Eqs. \eqref{1.1} and  \eqref{eq3.0}  respectively. %Theorem \ref{thm1} and the proof of Propositions \ref{prop3.1}
\end{coro}

\subsection{Symplectic forms} A non-degenerate skew-symmetric bilinear form $\B:A\times A\ra \F$ is said to be \textbf{symplectic} if $\B(xy,z)+\B(yz, x)+\B(zx, y)=0$ for all $x,y,z\in A$. %; see \cite[Section 5]{Gon12} for more details.

\begin{lem}\label{lem3.4}
Let $V$ be a vector space over $\F$ and $r\in V\otimes V$ be non-degenerate. Then $r$ is skew-symmetric if and only if the bilinear form $\B_r:V\times V\ra\F$ defined by $(x, y)\mapsto \pair{T_r^{-1}(x),y}$ is skew-symmetric, where $T_r$ is defined as in Eq. \eqref{1.1}.
\end{lem}

\begin{proof}
Note that $r$ is non-degenerate, thus for any $x,y\in V$, there exist unique $\xi,\eta\in V^*$ such that $x=T_r(\xi)$ and $y=T_r(\eta)$.
Now we assume that $r$ is skew-symmetric. Then $\B_r(x,y)+\B_r(y,x)=\pair{T_r^{-1}(x),y}+\pair{T_r^{-1}(y),x}=
\pair{\xi,T_r(\eta)}+\pair{\eta,T_r(\xi)}=0$, where the last equality follows from Lemma \ref{lem2.7}. Hence, $\B_r$ is skew-symmetric. Conversely, by Lemma \ref{lem2.7}, it suffices to show that $\pair{\xi,T_r(\eta)}+\pair{\eta, T_r(\xi)}=0$ for all $\xi,\eta\in V^*$. In fact, $\pair{\xi,T_r(\eta)}+\pair{\eta,T_r(\xi)}=\pair{T_r^{-1}(T_r(\xi)),T_r(\eta)}+\pair{T_r^{-1}(T_r(\eta)),T_r(\xi)}=
\B_r(T_r(\xi),T_r(\eta))+\B_r(T_r(\eta),T_r(\xi))=\B_r(x,y)+\B_r(y,x)=0$, since $\B_r$ is skew-symmetric.
\end{proof}

\begin{proof}[Proof of Theorem \ref{thm2}]
Suppose that $x,y,z\in A$ are arbitrary elements. Note that $r$ is non-degenerate, thus there exist unique
$\xi,\eta\in A^*$ such that $x=T_r(\xi)$ and $y=T_r(\eta)$.

($\RA$) As $r\in\Sol(A)$ is skew-symmetric, it follows from Theorem
\ref{thm1} that
\begin{eqnarray*}
\B_r(xy,z)&=&\pair{T_r^{-1}(xy),z}=\pair{T_r^{-1}(T_r(\xi)T_r(\eta)),z}\\
&=&\pair{T_r^{-1}(T_r(\ad^*_{T_r(\xi)}(\eta)-\ad^*_{T_r(\eta)}(\xi)),z} \\
&=&\pair{\ad^*_{T_r(\xi)}(\eta)-\ad^*_{T_r(\eta)}(\xi),z}\\
&=&\pair{\ad^*_{T_r(\xi)}(\eta),z}-\pair{\ad^*_{T_r(\eta)}(\xi),z}.
\end{eqnarray*}
However, $\pair{\ad^*_{T_r(\xi)}(\eta),z}=-\pair{\eta,\ad_{T_r(\xi)}(z)}=-\pair{\eta,T_r(\xi)z}=-\pair{T_r^{-1}(y),xz}=-\B_r(y,xz)$. Similarly,
$\pair{\ad^*_{T_r(\eta)}(\xi),z}=-\B_r(x,yz)$. Thus $\B_r(xy,z)=\B_r(x,yz)-\B_r(y,xz)=\B_r(x,yz)+\B_r(y,zx)$.
Lemma \ref{lem3.4} asserts that $\B_r$ is skew-symmetric. Hence
$\B_r(xy,z)+\B_r(yz,x)+\B_r(zx,y)=0$, that is, $\B_r$ is a symplectic form.

($\LA$) Now we assume that $\B_r$ is a symplectic form. By Theorem
\ref{thm1}, it suffices to show that $T_r\in \OO_A(A^*,\ad^*)$.
We have seen from the previous proof that
$\B_r(xy,z)=\pair{T_r^{-1}(xy),z}=\pair{T_r^{-1}(T_r(\xi)T_r(\eta)),z}$,
$\B_r(yz, x)=-\B_r(x,yz)=\pair{\ad^*_{T_r(\eta)}(\xi),z}$
and $\B_r(zx,y)=-\B_r(xz,y)=\B_r(y,xz)=-\pair{\ad^*_{T_r(\xi)}(\eta),z}$.
Thus
\begin{eqnarray*}
0&=&\B_r(xy,z)+\B_r(yz, x)+\B_r(zx, y)\\
&=&\pair{T_r^{-1}(T_r(\xi)T_r(\eta)),z}+\pair{\ad^*_{T_r(\eta)}(\xi),z}-\pair{\ad^*_{T_r(\xi)}(\eta),z}\\
&=&\pair{T_r^{-1}(T_r(\xi)T_r(\eta))+\ad^*_{T_r(\eta)}(\xi)-\ad^*_{T_r(\xi)}(\eta),z}
\end{eqnarray*}
for all $z\in A$.
As the natural pairing on $A$ is non-degenerate, it follows that $T_r^{-1}(T_r(\xi)T_r(\eta))+\ad^*_{T_r(\eta)}(\xi)-\ad^*_{T_r(\xi)}(\eta)=0$,
i.e., $T_r^{-1}(T_r(\xi)T_r(\eta))=\ad^*_{T_r(\xi)}(\eta)-\ad^*_{T_r(\eta)}(\xi)$. Multiplying the two sides of this equation with $T_r$, we see that $T_r(\xi)T_r(\eta)=T_r(\ad^*_{T_r(\xi)}(\eta)-\ad^*_{T_r(\eta)}(\xi))$.
This means that $T_r\in \OO_A(A^*,\ad^*)$, and therefore, $r\in \Sol(A)$, as desired.
\end{proof}

\begin{exam}\label{exam3.5}
{\rm
Let $A$ be the 4-dimensional Malcev algebra  with the basis $\{e_1,\dots,e_4\}$ defined in Example \ref{exam2.1}.
With respect to this basis, a direct calculation shows that any symplectic form $\B$ on $A$ has  the following form
$$\begin{pmatrix}
   0   & a&b&c   \\
   -a   & 0&-c&d  \\
    -b & c&0&e  \\
  -c     & -d&-e&0   \\
\end{pmatrix},$$
where $a,b,c,d,e\in\F$ and $c^2-ae\pm bd\neq0.$ By Theorem \ref{thm2}, one can construct all non-degenerate skew-symmetric solutions of the CYBE on $A$.
For instance, if we set $c=1$ and $a=b=d=e=0$, then  $r=e_1\otimes e_4-e_4\otimes e_1-e_2\otimes e_3+e_3\otimes e_2$ is   a non-degenerate skew-symmetric solution.
\hbo}\end{exam}

\subsection{The CYBE and semi-direct products}
To give a proof of Theorem \ref{thm3}, we use notations appeared in Introduction.
The identification of $A\otimes V^*$ with $\Hom(V,A)$ can be realized via the  linear isomorphism $\tau$ defined by
sending $x\otimes \xi$ to $\tau_{x\otimes \xi}$, where $x\in A, \xi\in V^*$ and $\tau_{x\otimes \xi}(v):=\xi(v)x$ for all $v\in V$.
Let  $\{v_1,\dots,v_n\}$ be a basis  of $V$ and  $\{\xi_1,\dots,\xi_n\}$ be the dual basis of $V^*$. Then we can identify an element $T\in\Hom(V,A)$ with the element
\begin{equation}
\label{eq3.1}
\wt{T}:=\sum_{i=1}^n T(v_i)\otimes \xi_i\in A\otimes V^*\subseteq (A\oplus V^*)\otimes (A\oplus V^*).
\end{equation}
We  define
\begin{equation}
\label{eq3.2}
r_T:=\wt{T}-\sigma(\wt{T})=\sum_{i=1}^n (T(v_i)\otimes \xi_i-\xi_i\otimes T(v_i))\in (A\oplus V^*)\otimes (A\oplus V^*).
\end{equation}

To analyze equivalent conditions of $r_T$ being a solution of the CYBE on $A\ltimes_{\rho^*}V^*$, we first note that
\begin{eqnarray}\label{eq3.3}
(r_T)_{12}(r_T)_{13}&=&\sum_{i,j=1}^n(T(v_i)T(v_j)\otimes \xi_i\otimes \xi_j-\rho^{*}(T(v_i))\xi_j\otimes \xi_i\otimes T(v_j)\\
&&\;\;\;\;\;\;\;\;+\rho^{*}(T(v_j))\xi_i\otimes T(v_i)\otimes \xi_j)\nonumber.
\end{eqnarray}
As $\rho^{*}(T(v_i))\xi_j\in V^*$ for all $i,j\in\{1,\dots,n\}$, we  assume that $\rho^{*}(T(v_i))\xi_j=a_1(ij) \xi_1+\dots+a_n(ij) \xi_n$, where
$a_s(ij)\in\F.$ For any $k\in\{1,\dots,n\}$, we have $a_k(ij)  =  \pair{\sum_{s=1}^n a_s(ij)\xi_s,v_k} = \pair{\rho^*(T(v_i))\xi_j,v_k}=-\pair{\xi_j,\rho(T(v_i))v_k}.$ Thus
\begin{equation}
\label{eq3.4}
\rho^*(T(v_i))\xi_j=-\sum_{k=1}^n\pair{\xi_j,\rho(T(v_i))v_k}\xi_k.
\end{equation}
Similarly, we observe that
\begin{equation}
\label{eq3.5}
\rho(T(v_i))v_j=\sum_{k=1}^n\pair{\xi_k,\rho(T(v_i))v_j}v_k.
\end{equation} Hence, it follows from Eqs. (\ref{eq3.4}) and (\ref{eq3.5}) that
\begin{eqnarray*}
\sum_{i,j=1}^n\rho^*(T(v_i))\xi_j\otimes \xi_i\otimes T(v_j)
&=&\sum_{i,j=1}^n(-\sum_{k=1}^n\pair{\xi_j,\rho(T(v_i))v_k}\xi_k)\otimes \xi_i\otimes T(v_j)\\
&=&-\sum_{i,k=1}^n \xi_k\otimes \xi_i\otimes T(\sum_{j=1}^n\pair{\xi_j,\rho(T(v_i))v_k}v_j)\\
&=&-\sum_{i,j=1}^n \xi_j\otimes \xi_i\otimes T(\sum_{k=1}^n\pair{\xi_k,\rho(T(v_i))v_j}v_k)\\
&=&-\sum_{i,j=1}^n \xi_j\otimes \xi_i\otimes T(\rho(T(v_i))v_j).
\end{eqnarray*}
Further, a similar calculation shows that
\begin{eqnarray*}
\sum_{i,j=1}^n\rho^{*}(T(v_j))\xi_i\otimes T(v_i)\otimes \xi_j %=-\sum_{i,j=1}^n\rho^{*}(T(v_j))\xi_i\otimes T(v_i)\otimes \xi_j
=-\sum_{i,j=1}^n\xi_i\otimes T(\rho(T(v_j))v_i)\otimes \xi_j.
\end{eqnarray*}
Taking the previous two equations back to Eq. (\ref{eq3.3}), we see that
\begin{eqnarray*}
(r_T)_{12} (r_T)_{13}=\sum_{i,j=1}^n((T(v_i)T(v_j)\otimes \xi_i\otimes \xi_j+\xi_j\otimes \xi_i \otimes T(\rho(T(v_i))v_j)-\xi_i\otimes T(\rho(T(v_j))v_i)\otimes \xi_j).
\end{eqnarray*}
We proceed in this way on $(r_T)_{13}(r_T)_{23}$ and $(r_T)_{23}(r_T)_{12}$ and eventually  derive
\begin{eqnarray}\label{eq3.6}
&&(r_T)_{12}(r_T)_{13}+(r_T)_{13}(r_T)_{23}-(r_T)_{23}(r_T)_{12}\nonumber\\
&=&\sum_{i,j=1}^n (T(v_i)T(v_j)-T(\rho(T(v_i))v_j)+T(\rho(T(v_j))v_i))\otimes \xi_i\otimes \xi_j\\
&&+\sum_{i,j=1}^n \xi_i\otimes(T(v_j)T(v_i)-T(\rho(T(v_j))v_i)+T(\rho(T(v_i))v_j))\otimes \xi_j\nonumber\\
&&+\sum_{i,j=1}^n \xi_i\otimes \xi_j\otimes(T(v_i)T(v_j)-T(\rho(T(v_i))v_j)+T(\rho(T(v_j))v_i)).\nonumber
\end{eqnarray}

Now we are ready to give a proof to our third main theorem.

\begin{proof}[Proof of Theorem \ref{thm3}] ($\RA$) Assume that $T\in\OO_A(V,\rho)$ is an $\OO$-operator. Then
$T(v_i)T(v_j)-T(\rho(T(v_i))v_j)+T(\rho(T(v_j))v_i)=0$ for all $i,j\in\{1,\dots,n\}$. Thus the right-hand side of Eq. (\ref{eq3.6}) is zero. This implies that $(r_T)_{12}(r_T)_{13}+(r_T)_{13}(r_T)_{23}-(r_T)_{23}(r_T)_{12}=0$, i.e., $r_T$ is a solution of the CYBE on $A\ltimes_{\rho^*}V^*$.

($\LA$)
Suppose that $r_T\in \Sol(A\ltimes_{\rho^*}V^*)$ is a solution, that is $(r_T)_{12}(r_T)_{13}+(r_T)_{13}(r_T)_{23}-(r_T)_{23}(r_T)_{12}=0$. Thus the right-hand side of Eq. (\ref{eq3.6}) is equal to zero. Let  $\{x_1,\dots,x_m\}$ be a basis of $A$. Assume that $T(v_i)T(v_j)-T((\rho(T(v_i))v_j-\rho(T(v_j))v_i)=c_1(ij)x_1+\dots+c_m(ij)x_m$ for some $c_1(ij),\dots, c_m(ij)\in\F$. Hence,
$$0=\sum_{i,j=1}^n \sum_{k=1}^m c_k(ij)(x_k\otimes \xi_i\otimes \xi_j+\xi_j\otimes x_k\otimes \xi_i+\xi_i\otimes \xi_j\otimes x_k).$$
Since $\{x_k\otimes \xi_i\otimes \xi_j, \xi_j\otimes x_k\otimes \xi_i, \xi_i\otimes \xi_j\otimes x_k\mid 1\leqslant i,j\leqslant n,1\leqslant k\leqslant m\}$ is a
subset of a basis of $(A\ltimes_{\rho^*}V^*)^{\otimes 3}$, its elements are linearly independent over $\F$. Hence,
$c_k(ij)=0$ for all $k, i$ and $j$, which means that
\begin{eqnarray*}
T(v_i)T(v_j)=T(\rho(T(v_i))v_j-\rho(T(v_j))v_i),
\end{eqnarray*}
for all $i,j\in\{1,\dots,n\}$. Therefore, $T\in\OO_A(V,\rho)$ and the proof is completed.
\end{proof}

\begin{exam}\label{exam3.6}
{\rm We consider the 4-dimensional Malcev algebra $A$ with the basis $\{e_1,\dots, e_4\}$ and the skew-symmetric solution $r=e_1\otimes e_4-e_4\otimes e_1-e_2\otimes e_3+e_3\otimes e_2$ of the CYBE on $A$ described  in Example \ref{exam3.5}.  By Theorem \ref{thm1},  we see that the linear map $T: A^* \ra A$ defined by
\begin{equation}\label{eq3.8}
T(\varepsilon_1)=-e_4, T(\varepsilon_2)=e_3, T(\varepsilon_3)=-e_2, T(\varepsilon_4)=e_1
\end{equation}
is an $\OO$-operator of $A$ associated to the coadjoint representation $(A^*,\ad^*)$, where $\{\varepsilon_1,\varepsilon_2,\varepsilon_3,\varepsilon_4\}$ is the dual basis of $A^*$.

Let $(A^*)^*$ be the dual space of $A^*$ with a basis $\{x_1,\dots, x_4\}$. We identify $(A^*)^*$ with $A$ and thus $\{x_1,\dots, x_4\}$ could be viewed as another basis of $A$. Take the $\OO$-operator $T$ in Theorem \ref{thm3} as in Eq. (\ref{eq3.8}).
We note that all non-zero products in Malcev algebra $A\ltimes_{(\ad^*)^*}(A^*)^*=A\ltimes_{\ad}A$ are given by
\begin{eqnarray*}
e_1e_2=-e_2, e_1e_3=-e_3, e_1e_4=e_4, e_2e_3=2e_4, e_1x_2=-x_2, e_1x_3=-x_3, \\
e_1x_4=x_4, e_2x_3=2x_4, e_2x_1=x_2, e_3x_1=x_3, e_4x_1=-x_4, e_3x_2=-2x_4.
\end{eqnarray*}
 It follows from Theorem \ref{thm3} that $r_T=-e_4\otimes x_1+x_1\otimes e_4+e_3\otimes x_2-x_2\otimes e_3-e_2\otimes x_3+x_3\otimes e_2+e_1\otimes x_4-x_4\otimes e_1$ is a skew-symmetric solution of CYBE on $A\ltimes_{\ad}A$.
%The co-multiplication $$\Delta_{r_T}:A\ltimes_{\ad}A\ra (A\ltimes_{\ad}A)\otimes (A\ltimes_{\ad}A)$$ induced by $r_T$, together with the multiplication in $A\ltimes_{\ad}A$, gives rise to a Malcev bialgebra structure on $A\ltimes_{\ad}A$, where $\Delta_{r_T}(a):=r_Ta=-e_4a\otimes x_1+e_4\otimes ax_1+x_1a\otimes e_4-x_1\otimes ae_4+e_3a\otimes x_2-e_3\otimes ax_2-x_2a\otimes e_3+x_2\otimes ae_3-e_2a\otimes x_3+e_2\otimes ax_3+x_3a\otimes e_2-x_3\otimes ae_2+e_1a\otimes x_4-e_1\otimes ax_4+x_4a\otimes e_1-x_4\otimes ae_1$, for all $a\in A\ltimes_{\ad}A$.
\hbo}\end{exam}

\section{$\OO$-operators and the CYBE on Pre-Malcev Algebras}\label{sec4}
\setcounter{equation}{0}
\renewcommand{\theequation}
{4.\arabic{equation}}
\setcounter{theorem}{0}
\renewcommand{\thetheorem}
{4.\arabic{theorem}}

\noindent After recalling basic facts on bimodules of pre-Malcev algebras, we study connections between $\OO$-operators and compatible pre-Malcev structures on a Malcev algebra, giving a proof of Theorem \ref{thm4} with two applications. Comparing with Theorems \ref{thm1} and  \ref{thm3}, we also derive several analogous results on symmetric solutions of the CYBE on pre-Malcev algebras.

\subsection{$\OO$-operators of Malcev algebras and Pre-Malcev algebras}
Recall in \cite[Definition 4]{Mad17} that a pre-Malcev algebra $\A$ is a vector space over $\F$ endowed with a binary product $\cdot$  satisfying an identity $P_M(x,y,z,t)=0$, where
\begin{eqnarray}\label{eq4.1}
&&P_M(x,y,z,t)\nonumber\\
&=&(y\cdot z)\cdot(x\cdot t)-(z\cdot y)\cdot(x\cdot t)+((x\cdot y)\cdot z)\cdot t-((y\cdot x)\cdot z)\cdot t+(z\cdot(y\cdot x))\cdot t\nonumber\\
&&-(z\cdot(x\cdot y))\cdot t+y\cdot((x\cdot z)\cdot t)-y\cdot((z\cdot x)\cdot t)+z\cdot(x\cdot (y\cdot t))-x\cdot(y\cdot (z\cdot t))\qquad
\end{eqnarray}
for all $x,y,z,t\in\A$. As Malcev-admissible algebras, pre-Malcev algebras extend the notion of pre-Lie algebras (or left-symmetric algebras) which have been studied extensively; see for example \cite{Bai04,Man11,Seg92}.

%Note that pre-Malcev algebras generalize pre-Lie algebras, that is every pre-Lie algebra is a pre-Malcev algebra. %The existence of subadjacent Malcev algebras and compatible pre-Malcev algebras was given in \cite[Proposition 5]{Mad17}. Given a pre-Malcev algebra$(\A,\cdot)$, there is a Malcev algebra $[\A]$ defined by the commutator $xy=x\cdot y-y\cdot x$, and the left multiplication in $\A$ $\ell_x$ induces a representation of Malcev algebra $[\A]$. Conversely, it is true.

\begin{exam}\label{exam4.1}
{\rm
Let $A$ be the 4-dimensional Malcev algebra  appeared in Example \ref{exam2.1} with the basis $\{e_{1}, e_{2}, e_{3}, e_{4}\}$. A direct calculation verifies that the following non-zero noncommutative products:
$$e_{1}\cdot e_{2}=-e_{2},  e_{1}\cdot e_{3}=-e_{3},  e_{1}\cdot e_{4}=e_{4},  e_{2}\cdot e_{3}=2e_{4},$$ give rise to a compatible pre-Malcev algebra structure $\A$ on $A$. In other words, $[\A]=A$.
\hbo}\end{exam}

Let $(\A,\cdot)$ be a pre-Malcev algebra over $\F$. A triple $(V,\ell,\re)$ of a vector space $V$ over $\F$ and two linear maps $\ell, \re:\A\ra \End(V) (x\mapsto \ell_x, x\mapsto \re_x)$  is called a \textbf{bimodule} of $\A$ if the following four equations hold:
\begin{eqnarray}
\re_x \re_y \re_z-\re_x \re_y\ell_z-\re_x\ell_y \re_z+\re_x\ell_y\ell_z-\re_{z\cdot(y\cdot x)}+\ell_y \re_{z\cdot x}+\ell_{z\cdot y}\re_x-\ell_{y\cdot z}\re_x-\ell_z \re_x\ell_y+\ell_z \re_x \re_y&=& 0,\qquad\label{eq4.2}\\
\re_{x}\re_{y}\ell_{z}-\re_{x}\re_{y}\re_{z}-\re_{x}\ell_{y}\ell_{z}+\re_{x}\ell_{y}\re_{z}-\ell_{z}\re_{y\cdot x}+\ell_{y}\ell_{z}\re_{x}+\re_{z\cdot x}\re_{y}
-\re_{z\cdot x}\ell_{y}-\re_{(y\cdot z)\cdot x}+\re_{(z\cdot y)\cdot x}&=& 0,\qquad\label{eq4.3}\\
\re_{x}\ell_{y\cdot z}-\re_{x}\ell_{z\cdot y}-\re_{x}\re_{y\cdot z}+\re_{x}\re_{z\cdot y}-\ell_{y}\ell_{z}\re_{x}+\re_{y\cdot(z\cdot x)}+\re_{y\cdot x}\ell_{z}
-\re_{y\cdot x}\re_{z}-\ell_{z}\re_{x}\re_{y}+\ell_{z}\re_{x}\ell_{y}&=& 0,\qquad\label{eq4.4}\\
\ell_{(x\cdot y)\cdot z}-\ell_{(y\cdot x)\cdot z}-\ell_{z\cdot(x\cdot y)}+\ell_{z\cdot(y\cdot x)}-\ell_{x}\ell_{y}\ell_{z}+\ell_{z}\ell_{x}\ell_{y}
+\ell_{y\cdot z}\ell_{x}-\ell_{z\cdot y}\ell_{x}-\ell_{y}\ell_{z\cdot x}+\ell_{y}\ell_{x\cdot z}&=& 0,\qquad\label{eq4.5}
\end{eqnarray}
where $x,y,z\in \A$. %, where $\End(V)$ denotes the general associative matrix algebra over $\F$.
Equivalently,
a triple $(V,\ell,\re)$ is an $\A$-bimodule if and only if the direct sum $\A\oplus V$ of vector spaces is turned into a pre-Malcev algebra, called the \textbf{semi-direct product} of $\A$ and $V$ via $(\ell, \re)$, by defining the binary product on $\A\oplus V$ as
\begin{equation}\label{eq4.6}
(x,u)\cdot(y,v):=(x\cdot y,\ell_x(v)+\re_y(u))
\end{equation}
for all $x,y\in\A$ and $u,v\in V$. We denote this pre-Malcev algebra by $\A\ltimes_{\ell,\re}V$.

Suppose that $\ell$ and $\re$ are two linear maps from $\A$ to $\End(V)$. Consider the dual space $V^*$ of $V$ and $\End(V^*)$. We define two linear maps $\ell^*,\re^*: \A\ra \End(V^*)$ by
\begin{equation}\label{eq4.7}
\pair{\ell^{*}_x(\xi),v}:=-\pair{\xi,\ell_x(v)}, \pair{\re^{*}_x(\xi),v}:=-\pair{\xi,\re_x(v)},
\end{equation}
respectively, where $x\in\A,\xi\in V^*$ and $v\in V$. Moreover, if $(V,\ell,\re)$ is an $\A$-bimodule, then one can show via a direct check that $(V^*,\ell^*-\re^*,-\re^*)$ is also an $\A$-bimodule.

Let $(\A, \cdot)$ be a pre-Malcev algebra. For  elements $x, y\in \A$, the left multiplication operator $L_x$ is  defined in the Introduction and we also define the right multiplication operator $R_x(y):=y\cdot x$. Let $L: \A\ra\End(\A)$ with $x\mapsto L_x$ and $R: \A\ra\End(\A)$ with $x\mapsto R_x$ for all $x\in \A$ be two linear maps.
Then $(\A,L,R)$ is an $\A$-bimodule and hence $(\A^*,L^*-R^*,-R^*)$ is also an $\A$-bimodule.

A bimodule of a pre-Malcev algebra can be used to construct representations of the subadjacent Malcev algebra. In fact, if  $(V,\ell,\re)$ is a bimodule of a pre-Malcev algebra $\A$, then it can be checked directly that $(V,\ell)$ and $(V,\ell-\re)$ are both representations of the Malcev algebra $[\A]$.
%(\cite[Theorem 2.2]{AX18}).

%\subsection{$\OO$-operators and  pre-Malcev algebras}
Before giving a proof of Theorem \ref{thm4}, we first reveal a general connection between
$\OO$-operators of Malcev algebras and  pre-Malcev algebras, generalizing a link between Rota-Baxter operators and
pre-Malcev algebras (\cite[Proposition 9]{Mad17}); also see \cite[Section 3]{Bai07} for the case of Lie algebras.

\begin{prop}\label{prop4.3}
Let $(V,\rho)$ be a representation of a Malcev algebra $A$. Given a $T\in\OO_A(V,\rho)$, we define
a binary product on $V$ by $v\ast w:=\rho(T(v))w$ for all $v,w\in V$. Then the following results hold.
\begin{enumerate}
  \item $(V,\ast)$ is a pre-Malcev algebra.
  \item The binary product
  \begin{equation}\label{eq:4.8}
  T(v)\cdot T(w):=T(v\ast w)
  \end{equation}
 gives rise to a pre-Malcev algebra structure on $T(V):=\{T(v)\mid v\in V\}\subseteq A.$
  \item In particular, if $T$ is surjective, then there exists a compatible pre-Malcev algebra structure $\A_T$ on $A$.
\end{enumerate}
\end{prop}

\begin{proof} (1) A direct verification on Eq. (\ref{eq4.1}) applies to a proof of the first statement. 

(2) We first show that Eq. \eqref{eq:4.8} is well-defined. We only need to verify that, if $T(v)=0$, then $T(v\ast w)=T(w\ast v)=0$ for all $w\in V$. In fact, if $T(v)$=0, then for all $w\in W$, we have
\begin{eqnarray*}
T(v\ast w)&=&T(\rho(T(v))w)=0,\\
T(w\ast v)&=& T(\rho(T(w))v)=T(\rho(T(w))v-\rho(T(v))w)=T(w)T(v)=0.
\end{eqnarray*}
Hence  Eq.~\eqref{eq:4.8} is well-defined.
Then it is direct to verify Eq. (\ref{eq4.1}) on $T(V)$.

\delete{show that $T$ preserves pre-Malcev algebra structures.  
To verify Eq. (\ref{eq4.1}) on $T(V)$,
for all  $T(u),T(v),T(w),T(t)\in T(V)$, where $u,v,w,t\in V$, we note that
\begin{eqnarray*}
&&P_M(T(u),T(v),T(w),T(t))\\
&=&(T(v)\cdot T(w))\cdot(T(u)\cdot T(t))-(T(w)\cdot T(v))\cdot(T(u)\cdot T(t))+((T(u)\cdot T(v))\cdot T(w))\cdot T(t)\\
&&-((T(v)\cdot T(u))\cdot T(w))\cdot T(t)+(T(w)\cdot(T(v)\cdot T(u)))\cdot T(t)-(T(w)\cdot(T(u)\cdot T(v)))\cdot T(t)\\
&&+T(v)\cdot((T(u)\cdot T(w))\cdot T(t))-T(v)\cdot((T(w)\cdot T(u))\cdot T(t))\\
&&+T(w)\cdot(T(u)\cdot(T(v)\cdot T(t)))-T(u)\cdot(T(v)\cdot (T(w)\cdot T(t)))\\
&=&T((v\ast w)\ast(u\ast t)-(w\ast v)\ast(u\ast t)+((u\ast v)\ast w)\ast t-((v\ast u)\ast w)\ast t+(w\ast(v\ast u))\ast t\\
&&-(w\ast(u\ast v))\ast t+v\ast((u\ast w)\ast t)-v\ast((w\ast u)\ast t)+w\ast(u\ast(v\ast t))-u\ast(v\ast(w\ast t)))\\
&=&T(P_M(u,v,w,t))=0.
\end{eqnarray*}
%The last equation follows from the first statement.
}

(3) By the second statement,  we obtain that  $A=T(V)$ has a pre-Malcev algebra structure $\A_T$. It is sufficient to show that the pre-Malcev algebra $\A_T$ is compatible with $A$. In fact, $T(v)\cdot T(w)-T(w)\cdot T(v)=T(v\ast w)-T(w\ast v)=T(\rho(T(v))w)-T(\rho(T(w))v)=T(\rho(T(v))w-\rho(T(w))v)=T(v)T(w)$, where the last equation follows from the assumption that $T\in\OO_A(V,\rho)$. Hence, $[\A_T]=A$, i.e., $\A_T$ is a compatible pre-Malcev algebra structure  on $A$.
\end{proof}

\begin{proof}[Proof of Theorem \ref{thm4}]
Since $T\in\OO_A(V,\rho)$ is invertible, for $x,y\in A$, there exist unique $v,w\in V$ such that $x=T(v)$ and $y=T(w)$. By the third statement of Proposition \ref{prop4.3}, there exists a compatible pre-Malcev algebra on $A$ defined by
\begin{eqnarray*}
x\cdot y= T(v)\cdot T(w)=T(v\ast w)=T(\rho(T(v))w)=T(\rho(x)T^{-1}(y)).
\end{eqnarray*}

Conversely, %to show the identity map $\id_A\in \OO_A(\A, L)$, note that
 we have $\id_A(L_{\id_A(x)}y-L_{\id_A(y)}x)=\id_A(L_x(y)-L_y(x))=\id_A(x\cdot y-y\cdot x)=\id_A(xy)=xy=\id_A(x)\id_A(y)$ for  all $x,y\in A$, which means $\id_A\in\OO_A(\A, L)$.
\end{proof}

Theorem \ref{thm4} has the following two direct applications for which the first one gives a way to construct a skew-symmetric solution of the CYBE on the semi-direct product of a Malcev algebra $A$ and its representation $(\A^*,L^{*})$; and the second one shows that
a Malcev algebra admitting a non-degenerate symplectic form $\B$ must have a compatible
pre-Malcev algebra structure; compared with \cite{Bau99} for the case of Lie algebras.

\begin{coro}\label{coro4.4}
Let $\A$ be a pre-Malcev algebra with a basis $\{e_1,\dots, e_n\}$ and $\{\varepsilon_1,\dots, \varepsilon_n\}$ be the basis of $A^*$ dual to
$\{e_1,\dots, e_n\}$. Then the element
$$r:=\sum_{i=1}^n (e_i\otimes \varepsilon_i-\varepsilon_i\otimes e_i)$$
is a skew-symmetric solution of the CYBE on the Malcev algebra $[\A]\ltimes_{L^*}\A^*$.
\end{coro}

\begin{proof} Consider the identity map $\id_A$. By Theorem \ref{thm4}, we see that $\id_A$ is an $\OO$-operator of $[\A]$ associated to $(\A,L)$.
It follows from Eq. (\ref{eq3.1}) that $\wt{\id_A}=\sum_{i=1}^n \id_A(e_i)\otimes \varepsilon_i=\sum_{i=1}^ne_i\otimes \varepsilon_i$.
Thus it follows from Theorem \ref{thm3} that $r=\sum_{i=1}^n (e_i\otimes \varepsilon_i-\varepsilon_i\otimes e_i)$ is a skew-symmetric solution of the CYBE on the Malcev algebra $[\A]\ltimes_{L^*}\A^*$, as desired.
\end{proof}

\begin{prop}\label{prop4.5}
Let $A$ be a Malcev algebra admitting a non-degenerate symplectic form $\B.$ Then there exists a compatible
pre-Malcev algebra $(\A, \cdot)$ on $A$ such that $\B(x\cdot y,z)=-\B(y,xz)$ for all $x,y,z\in A$.
\end{prop}

\begin{proof}
Since $\B$ is a non-degenerate symplectic form,  we  define an invertible linear map $T: A^*\ra A$
by $\pair{T^{-1}(x),y}=\B(x,y)$ for $x,y\in A$. %to be the inverse of $\varphi_\B$ in Eq. (\ref{eq3.0}).
A similar argument as in the proof of Theorem \ref{thm2} shows that $T\in\OO_A(A^*,\ad^*)$.
By Theorem \ref{thm4}, there exists a compatible pre-Malcev algebra $\A$ given by
$x\cdot y=T(\ad^*_x(T^{-1}(y)))$ for all $x,y\in A$. %We  choose a natural pairing on $A$ such that $\pair{\B_x,y}=\B(x,y)$, where $x,y\in A$ and $\B_x$ is defined as Eq. (\ref{eq3.0}).
Hence we derive
\begin{eqnarray*}
\B(x\cdot y,z)=\B(T(\ad^*_x(T^{-1}(y))),z)=\pair{\ad^*_x(T^{-1}(y)),z}=-\pair{T^{-1}(y), xz}=-\B(y,xz)
\end{eqnarray*}
for all $x,y,z\in A.$ This completes the proof.
\end{proof}

\subsection{The CYBE on pre-Malcev algebras}

Let $(\A,\cdot)$ be a pre-Malcev algebra over $\F$ and $(V,\ell,\re)$ be an $\A$-bimodule. A linear map $T: V \ra\A$
is called an \textbf{$\mathcal {O}$-operator} of
$\A$ associated to $(V,\ell,\re)$ if
\begin{equation}\label{eq4.8}
T(v)\cdot T(w) = T(\ell_{T(v)}(w)+\re_{T(w)}(v))
\end{equation}
for all $v, w \in V$. We  write $\OO_\A(V,\ell,\re)$ for the set of all $\OO$-operators of $\A$
associated to $(V,\ell,\re)$.
We say that an element $r=\sum_{i}x_{i}\otimes y_{i}\in \A\otimes \A$ is a solution of the CYBE on $\A$ if
$-r_{12}\cdot r_{13}+r_{12}\cdot r_{23}+r_{13}r_{23}=0$, where
 $$r_{12}\cdot r_{13}=\sum_{i,j}x_{i}\cdot x_j\otimes y_{i}\otimes y_j,   r_{12}\cdot r_{23}=\sum_{i,j}x_{i}\otimes y_i\cdot x_j\otimes y_j,$$
 $$ r_{13}r_{23}=\sum_{i,j}x_{i}\otimes x_j\otimes y_iy_j=\sum_{i,j}x_{i}\otimes x_j\otimes (y_i\cdot  y_j-y_j\cdot  y_i),$$
denote the images of $r$ under the three standard embeddings from $\A\otimes\A$ to $\A^{\otimes 3}$ respectively.
We use $\Sol(\A)$ to denote the set of all solutions of the CYBE on $\A$. We first obtain an
analogue of Theorem \ref{thm1} in the case of pre-Malcev algebras and symmetric solutions of the CYBE.

\begin{thm}\label{thm4.6}
Let $\A$ be a finite-dimensional pre-Malcev algebra over a field $\F$ of characteristic zero and $r$ be a symmetric element in $\A\otimes \A$. Then $r\in \Sol(\A)$ if and only if $T_r\in \OO_\A(\A^*,L^*-R^*,-R^*)$.
\end{thm}

\begin{proof} Suppose that $r=\sum_{i}x_{i}\otimes y_{i}\in\A\otimes\A$ is symmetric. For all $\xi,\eta,\zeta\in \A^*$, we consider the natural pairing on $\A^{\otimes 3}$, and obtain
\begin{eqnarray*}
\pair{\xi\otimes\eta\otimes\zeta,r_{12}\cdot r_{13}}&=&\sum_{i,j}\pair{\xi\otimes\eta\otimes\zeta, x_{i}\cdot x_j\otimes y_{i}\otimes y_j}\\
&=&\sum_{i,j}\pair{\xi,x_i\cdot x_j}\pair{\eta,y_i}\pair{\zeta,y_j}\\
&=&\sum_{i,j}\pair{\xi,\pair{\eta, y_{i}}x_{i}\cdot\pair{\zeta,y_j}x_j}\\
&=&\pair{\xi,T_r(\eta)\cdot T_r(\zeta)}.
\end{eqnarray*}
Similarly, we have $\pair{\xi\otimes\eta\otimes\zeta, r_{12}\cdot r_{23}}=\pair{\eta, T_r(\xi)\cdot T_r(\zeta)}$ and $\pair{\xi\otimes\eta\otimes\zeta,r_{13}r_{23}}=\pair{\zeta, T_r(\xi)T_r(\eta)}$.
%Thus
%\begin{eqnarray}\label{eq4.9}
%&&\pair{\xi\otimes\eta\otimes\zeta,-r_{12}\cdot r_{13}+r_{12}\cdot r_{23}+r_{13}r_{23}}\nonumber\\
%&=&-\pair{\xi,T_r(\eta)\cdot T_r(\zeta)}+\pair{\zeta,T_r(\xi)T_r(\eta)}+\pair{\eta,T_r(\xi)\cdot T_r(\zeta)}.
%\end{eqnarray}
We  observe that
\begin{eqnarray*}
\pair{\xi, T_r((L^*_{T_r(\eta)}-R^*_{T_r(\eta)})(\zeta))}&=&\pair{(L^*_{T_r(\eta)}-R^*_{T_r(\eta)})(\zeta), T_r(\xi)}\\
&=&-\pair{\zeta, L_{T_r(\eta)}(T_r(\xi))}+\pair{\zeta, R_{T_r(\eta)}(T_r(\xi))}\\
&=&\pair{\zeta, T_r(\xi)\cdot T_r(\eta)-T_r(\eta)\cdot T_r(\xi)}\\
&=&\pair{\zeta, T_r(\xi)T_r(\eta)},
\end{eqnarray*}
and an analogous argument shows that $\pair{\xi, T_r(-R^*_{T_r(\zeta)}(\eta))}=\pair{\eta,T_r(\xi)\cdot T_r(\zeta)}.$
Hence
\begin{eqnarray}\label{eq4.10}
&&\pair{\xi, T_r(\eta)\cdot T_r(\zeta)-T_r((L^*_{ T_r(\eta)}-R^*_{ T_r(\eta)})(\zeta))-T_r(-R^*_{ T_r(\zeta)}(\eta))}\nonumber\\
&=&\pair{\xi,T_r(\eta)\cdot T_r(\zeta)}-\pair{\zeta,T_r(\xi)T_r(\eta)}-\pair{\eta,T_r(\xi)\cdot T_r(\zeta)}\\
&=&-\pair{\xi\otimes\eta\otimes\zeta,-r_{12}\cdot r_{13}+r_{12}\cdot r_{23}+r_{13}r_{23}}.\nonumber
\end{eqnarray}
$(\RA)$
Now we suppose $r\in \Sol(\A)$, that is, $-r_{12}\cdot r_{13}+r_{12}\cdot r_{23}+r_{13}r_{23}=0$. Thus it follows from Eq. (\ref{eq4.10}) that
$\pair{\xi, T_r(\eta)\cdot T_r(\zeta)-T_r((L^*_{ T_r(\eta)}-R^*_{ T_r(\eta)})(\zeta))-T_r(-R^*_{ T_r(\zeta)}(\eta))}=0$. Since $\xi$ is arbitrary, we see that $T_r(\eta)\cdot T_r(\zeta)-T_r((L^*_{ T_r(\eta)}-R^*_{ T_r(\eta)})(\zeta))-T_r(-R^*_{ T_r(\zeta)}(\eta))=0$, i.e. $T_r\in \OO_\A(\A^*,L^*-R^*,-R^*)$.
$(\LA)$ Conversely, the assumption that $T_r\in \OO_\A(\A^*,L^*-R^*,-R^*)$, together with Eq. (\ref{eq4.10}), implies that $\pair{\xi\otimes\eta\otimes\zeta,-r_{12}\cdot r_{13}+r_{12}\cdot r_{23}+r_{13}r_{23}}=0$. Thus $-r_{12}\cdot r_{13}+r_{12}\cdot r_{23}+r_{13}r_{23}=0$, i.e., $r\in \Sol(\A)$. The proof is completed.
\end{proof}

%As in Section \ref{sec1}, we identify $\Hom(V,\A)$ with $\A\otimes V^*$, and identify an element of $\A\otimes V^*$ with the image in $(\A\oplus V^*)\otimes (\A\oplus V^*)$ under the tensor product of the standard embeddings $\A\ra \A\oplus V^*$ and $V^*\ra \A\oplus V^*$. Given a linear map $T:V\ra \A$, we define a symmetric element $s_T$ as $s_T:=T+\sigma(T)\in (\A\oplus V^*)\otimes (\A\oplus V^*)$ via this two identifications.
The following result is a symmetric element version of Theorem \ref{thm3} for pre-Malcev algebras.

\begin{thm}\label{thm4.7}
Let $(V,\ell,\re)$ be a representation of a finite-dimensional pre-Malcev algebra $\A$ over a field $\F$ of characteristic zero and $T: V\ra\A$ be a linear map.
 Suppose that $\{v_1,\dots,v_n\}$ is a basis of $V$ and $\{\xi_1,\dots,\xi_n\}$ is the dual basis of $V^*$. Define
\begin{equation}
\label{eq4.11}
\wt{T}:=\sum_{i=1}^n T(v_i)\otimes \xi_i\in \A\otimes V^*\subseteq (\A\oplus V^*)\otimes (\A\oplus V^*).
\end{equation}   Then $T\in\OO_\A(V,\ell,\re)$ if and only if $s_T=\wt{T}+\sigma(\wt{T})\in \Sol(\A\ltimes_{\ell^*-\re^*,-\re^*}V^*)$. %, where $\wt{T}$ is defined by Eq. (\ref{eq4.11}) below.
\end{thm}

\begin{proof}
%The natural isomorphism $\tau$ defined by sending $x\otimes \xi$ to $\tau_{x\otimes \xi}$ leads us to identity $\A\otimes V^*$ with $\Hom(V,\A)$, where $x\in \A, \xi\in V^*$ and $\tau_{x\otimes \xi}(v):=\xi(v)x$ for all $v\in V$.
%Choose a basis $\{v_1,\dots,v_n\}$ for $V$ and a basis $\{\xi_1,\dots,\xi_n\}$ for $V^*$ dual to $\{v_1,\dots,v_n\}$. We identify an element $T\in\Hom(V,\A)$ with the element
%\begin{equation}\label{eq4.11}
%\wt{T}:=\sum_{i=1}^n T(v_i)\otimes \xi_i\in \A\otimes V^*\subseteq (\A\oplus V^*)\otimes (\A\oplus V^*).
%\end{equation}
%We define
%\begin{equation}\label{eq4.12}
%s_T=\wt{T}+\sigma(\wt{T})=\sum_{i=1}^n T(v_i)\otimes \xi_i+\sum_{i=1}^n \xi_i\otimes T(v_i)\in (\A\oplus V^*)\otimes(\A\oplus V^*).
%\end{equation}
Note that
\begin{eqnarray}\label{eq4.13}
\qquad-(s_T)_{12}\cdot(s_T)_{13}&=&\sum_{i, j=1}^n(-T(v_i)\cdot T(v_j)\otimes \xi_i\otimes \xi_j-(\ell^{*}_{T(v_i)}-\re^*_{T(v_i)})(\xi_j)\otimes \xi_i\otimes T(v_j)\qquad\\
&&\;\;\;\;\;\;\;\;-(-\re^*_{T(v_j)})(\xi_i)\otimes T(v_i)\otimes \xi_j)\nonumber.
\end{eqnarray}
As $(\ell^{*}_{T(v_i)}-\re^*_{T(v_i)})(\xi_j)\in V^*$ for all $i,j\in\{1,\dots,n\}$, we  assume that $(\ell^{*}_{T(v_i)}-\re^*_{T(v_i)})(\xi_j)=a_1(ij) \xi_1+\dots+a_n(ij) \xi_n$, where $a_s(ij)\in\F$. For any $k\in\{1,\dots,n\}$, we have
$$ a_k(ij)=\pair{\sum_{s=1}^n a_s(ij)\xi_s,v_k}=\pair{(\ell^{*}_{T(v_i)}-\re^*_{T(v_i)})(\xi_j),v_k}=-\pair{\xi_j,(\ell_{T(v_i)}-\re_{T(v_i)})(v_k)}.$$
Thus
\begin{equation}
\label{eq4.14}
(\ell^{*}_{T(v_i)}-\re^*_{T(v_i)})(\xi_j)=-\sum_{k=1}^n\pair{\xi_j,(\ell_{T(v_i)}-\re_{T(v_i)})(v_k)}\xi_k.
\end{equation}
Similarly,
\begin{equation}
\label{eq4.15}
\re^*_{T(v_i)}(\xi_j)=-\sum_{k=1}^n\pair{\xi_j,\re_{T(v_i)}(v_k)}\xi_k.
\end{equation} Hence, it follows from Eqs. (\ref{eq4.14}) and (\ref{eq4.15}) that
\begin{eqnarray*}
\sum_{i,j=1}^n(\ell^{*}_{T(v_i)}-\re^*_{T(v_i)})(\xi_j)\otimes \xi_i\otimes T(v_j)
&=&\sum_{i,j=1}^n\left(-\sum_{k=1}^n\pair{\xi_j,(\ell_{T(v_i)}-\re_{T(v_i)})(v_k)}\xi_k\right)\otimes \xi_i\otimes T(v_j)\\
&=&-\sum_{i,k=1}^n \xi_k\otimes \xi_i\otimes T\left(\sum_{j=1}^n\pair{\xi_j,(\ell_{T(v_i)}-\re_{T(v_i)})(v_k)}v_j\right)\\
&=&-\sum_{i,j=1}^n \xi_j\otimes \xi_i\otimes T\left(\sum_{k=1}^n\pair{\xi_k,(\ell_{T(v_i)}-\re_{T(v_i)})(v_j)}v_k\right)\\
&=&-\sum_{i,j=1}^n \xi_j\otimes \xi_i\otimes T((\ell_{T(v_i)}-\re_{T(v_i)})(v_j)),
\end{eqnarray*}
and
\begin{eqnarray*}
\sum_{i,j=1}^n \re^*_{T(v_j)}(\xi_i)\otimes T(v_i)\otimes \xi_j%=-\sum_{i,j=1}^n\rho^{*}(T(v_j))\xi_i\otimes T(v_i)\otimes \xi_j
=-\sum_{i,j=1}^n\xi_i\otimes T(\re_{T(v_j)}(v_i))\otimes \xi_j.
\end{eqnarray*}
Now Eq. (\ref{eq4.13}) reads
\begin{eqnarray*}
&&-(s_T)_{12}\cdot (s_T)_{13}\\
&=&\sum_{i,j=1}^n(-T(v_i)\cdot T(v_j)\otimes \xi_i\otimes \xi_j+\xi_j\otimes \xi_i \otimes T((\ell_{T(v_i)}-\re_{T(v_i)})(v_j))-\xi_i\otimes T(\re_{T(v_j)}(v_i))\otimes \xi_j).
\end{eqnarray*}
We do similar calculations on $(s_T)_{12}\cdot(s_T)_{23}$ and $(s_T)_{13}(s_T)_{23}$  and conclude
\begin{eqnarray}\label{eq4.16}
&&-(s_T)_{12}\cdot(s_T)_{13}+(s_T)_{12}\cdot(s_T)_{23}+(s_T)_{13}(s_T)_{23}\nonumber\\
&=&\sum_{i,j=1}^n (-T(v_i)\cdot T(v_j)+T(\ell_{T(v_i)}(v_j))+T(\re_{T(v_j)}(v_i)))\otimes \xi_i\otimes \xi_j\\
&&+\sum_{i,j=1}^n \xi_i\otimes(T(v_i)\cdot T(v_j)-T(\ell_{T(v_i)}(v_j))-T(\re_{T(v_j)}(v_i)))\otimes \xi_j\nonumber\\
&&+\sum_{i,j=1}^n \xi_i\otimes \xi_j\otimes(T(v_i)T(v_j)-T((\ell_{T(v_i)}-\re_{T(v_i)})(v_j))+T((\ell_{T(v_j)}-\re_{T(v_j)})(v_i))).\nonumber
\end{eqnarray}
We are ready to complete the proof.
($\RA$) We assume that $T\in\OO_\A(V,\ell,\re)$, that is,
$T(v_i)\cdot T(v_j)-T(\ell_{T(v_i)}(v_j))-T(\re_{T(v_j)}(v_i))=0$ for all $i,j\in\{1,\dots,n\}$. Thus the right-hand side of Eq. (\ref{eq4.16}) must be zero. This implies that $-(s_T)_{12}\cdot(s_T)_{13}+(s_T)_{12}\cdot(s_T)_{23}+(s_T)_{13}(s_T)_{23}=0$, i.e., $s_T$ is a solution of the CYBE on the pre-Malcev algebra $\A\ltimes_{\ell^*-\re^*,-\re^*}V^*$.
($\LA$)
Suppose that $s_T\in \Sol(\A\ltimes_{\ell^*-\re^*,-\re^*}V^*)$ is a solution of the CYBE. %Then $-(s_T)_{12}\cdot(s_T)_{13}+(s_T)_{12}\cdot(s_T)_{23}+(s_T)_{13}(s_T)_{23}=0$.
Thus the left-hand side of Eq. (\ref{eq4.16}) is equal to zero. We write $\{x_1,\dots,x_m\}$ for a basis of $\A$ and assume that $T(v_i)\cdot T(v_j)-T(\ell_{T(v_i)}(v_j)+\re_{T(v_j)}(v_i))=a_1(ij)x_1+\dots+a_m(ij)x_m$ for some $a_1(ij),\dots,a_m(ij)\in\F$. Hence, $$0=\sum_{i,j=1}^n \sum_{k=1}^m (-a_k(ij)x_k\otimes \xi_i\otimes \xi_j+a_k(ij)\xi_j\otimes x_k\otimes \xi_i+(a_k(ij)-a_k(ji))\xi_i\otimes \xi_j\otimes x_k).$$
The fact that $\{x_k\otimes \xi_i\otimes \xi_j, \xi_j\otimes x_k\otimes \xi_i, \xi_i\otimes \xi_j\otimes x_k\mid 1\leqslant i,j\leqslant n,1\leqslant k\leqslant m\}$ is a subset of a basis of $(\A\ltimes_{\ell^*-\re^*,-\re^*}V^*)^{\otimes 3}$ implies that these elements are linearly independent over $\F$.
Hence, $a_k(ij)=0$ for all $k,i$ and $j$. Hence,
$T(v_i)\cdot T(v_j)=T((\ell_{T(v_i)}(v_j)-\re_{T(v_j)}(v_i))$
for all $i,j\in\{1,\dots,n\}$. This shows that $T\in\OO_\A(V,\ell,\re)$ and we are done.
\end{proof}

\begin{coro}
Let $\A$ be a pre-Malcev algebra with a basis $\{e_1,\dots, e_n\}$ and $\{\varepsilon_1,\dots, \varepsilon_n\}$ be the basis of $\A^*$ dual to $\{e_1,\dots, e_n\}$.
Then the element
$$s:=\sum_{i=1}^n (e_i\otimes \varepsilon_i+\varepsilon_i\otimes e_i)$$
is a symmetric solution of the CYBE on the pre-Malcev algebra $\A\ltimes_{L^*,0}\A^*$.
\end{coro}

\begin{proof}
Since $\id_\A$  is an $\OO$-operator of $\A$ associated to the bimodule $(\A,L,0)$, we have $\wt{\id_\A}=\sum_{i=1}^n \id_\A(e_i)\otimes \varepsilon_i=\sum_{i=1}^ne_i\otimes \varepsilon_i$.
Thus it follows from Theorem \ref{thm4.7} that $s=\sum_{i=1}^n (e_i\otimes \varepsilon_i+\varepsilon_i\otimes e_i)$ is a symmetric solution of the CYBE on the pre-Malcev algebra $\A\ltimes_{L^*,0}\A^*$.
\end{proof}

We close this subsection by establishing connections between invertible $\OO$-operators and bilinear forms on a given pre-Malcev algebra $\A$.

\begin{prop}\label{prop4.8}
Let $\A$ be a pre-Malcev algebra and $T: \A^*\ra \A$ be an invertible linear map. Suppose $\B:\A\times \A\ra\F$ is a
bilinear form defined by $\B(x, y)= \pair{T^{-1}(x),y}$. For all $x,y,z\in \A$, we have the following results:
\begin{enumerate}
  \item $T\in\OO_\A(\A^*,L^*-R^*,0)$ if and only if  $\B(x\cdot y,z)=-\B(y,x\cdot z-z\cdot x)$.
  \item $T\in\OO_\A(\A^*,L^*-R^*,-R^*)$ if and only if $\B(x\cdot y,z)=-\B(y,x\cdot z)+\B(y,z\cdot x)+\B(x,z\cdot y)$.
  \item $T\in\OO_{[\A]}(\A^*,L^*)$ if and only if  $\B(xy,z)=\B(x,y\cdot z)-\B(y,x\cdot z)$.
\end{enumerate}
\end{prop}

\begin{proof} We first note that $T$ is invertible, thus for $x, y\in \A$, there exist unique $\xi,\eta\in \A^*$ such that $x=T(\xi)$ and $y=T(\eta)$.
(1) Assume that $T$ is an $\OO$-operator of $\A$ associated to the bimodule $(\A^*,L^*-R^*,0)$. We note that
\begin{eqnarray*}
\B(x\cdot y,z)+\B(y,x\cdot z-z\cdot x)
&=&\pair{T^{-1}(x\cdot y),z}+\pair{T^{-1}(y),x\cdot z-z\cdot x}\\
&=&\pair{T^{-1}(T(\xi)\cdot T(\eta)),z}+\pair{\eta,L_x(z)-R_x(z)}\\
&=&\pair{T^{-1}(T(\xi)\cdot T(\eta))-(L^*_{ x}-R^*_{ x})(\eta),z}\\
&=&\pair{T^{-1}(T(\xi)\cdot T(\eta)-T((L^*_{ x}-R^*_{ x})(\eta))),z}\\
&=&\pair{T^{-1}(T(\xi)\cdot T(\eta)-T((L^*_{ T(\xi)}-R^*_{ T(\xi)})(\eta))),z}=0
\end{eqnarray*}
Hence, $\B(x\cdot y,z)=-\B(y,x\cdot z-z\cdot x)$. Conversely, suppose that $\B(x\cdot y,z)=-\B(y,x\cdot z-z\cdot x)$. Then $0=\B(x\cdot y,z)+\B(y,x\cdot z-z\cdot x)=\pair{T^{-1}(T(\xi)\cdot T(\eta)-T((L^*_{ T(\xi)}-R^*_{ T(\xi)})(\eta))),z}.$
Since $z$ is arbitrary, we see that $T^{-1}(T(\xi)\cdot T(\eta)-T((L^*_{ T(\xi)}-R^*_{ T(\xi)})(\eta)))=0$.
Therefore, $T_r\in\OO_\A(\A^*,L^*-R^*,0)$.

(2) Suppose that $T\in\OO_\A(\A^*,L^*-R^*,-R^*)$, then $x\cdot y=T(\xi)\cdot T(\eta)=T((L^*_{ T(\xi)}-R^*_{ T(\xi)})(\eta)-R^*_{ T(\eta)}(\xi))$. We obtain that
\begin{eqnarray}\label{eq:4.16}
\B(x\cdot y,z)&=&\pair{T^{-1}(x\cdot y),z}=\pair{(L^*_{ T(\xi)}-R^*_{ T(\xi)})(\eta)- R^*_{ T(\eta)}(\xi),z}\nonumber\\
&=&\pair{(L^*_{ x}-R^*_{ x})(T^{-1}(y))-R^*_{ y}(T^{-1}(x)),z}\\
&=&-\pair{T^{-1}(y),(L_x-R_x)(z)}+\pair{T^{-1}(x),R_{y}(z)}\nonumber\\
&=&-\B(y,x\cdot z)+\B(y,z\cdot x)+\B(x,z\cdot y).\nonumber
\end{eqnarray}
Conversely, assume that $\B(x\cdot y,z)=-\B(y,x\cdot z)+\B(y,z\cdot x)+\B(x,z\cdot y)$. %Then
%\begin{eqnarray*}
%0&=&\B(x\cdot y,z)+\B(y,x\cdot z)-B(y,z\cdot x)-B(x,z\cdot y)\\
%&=&\pair{T^{-1}(x\cdot y),z}+\pair{T^{-1}(y),x\cdot z}-\pair{T^{-1}(y),z\cdot x}-\pair{T^{-1}(x),z\cdot y}\\
%&=&\pair{T^{-1}(T(\xi)\cdot T(\eta)),z}+\pair{\eta,x\cdot z}-\pair{\eta,z\cdot x}-\pair{\xi,z\cdot y}\\
%&=&\pair{T^{-1}(T(\xi)\cdot T(\eta)),z}+\pair{\eta,L_x(z)-R_x(z)}-\pair{\xi,R_{ y}(z)}\\
%&=&\pair{T^{-1}(T(\xi)\cdot T(\eta)),z}-\pair{(L^*_{ x}-R^*_{ x})(\eta),z}+\pair{R^*_{ y}(\xi),z}\\
%&=&\pair{T^{-1}(T(\xi)\cdot T(\eta)-T((L^*_{ T(\xi)}-R^*_{ T(\xi)})(\eta)-R^*_{ T(\eta)}(\xi))),z}.
%\end{eqnarray*}
%Since $z$ is arbitrary and the natural pairing on $\A$ is non-degenerate,
By a  discussion  similar to Eq. \eqref{eq:4.16}, we see that %$T_{r}^{-1}(T_{r}(\xi)\cdot T_{r}(\eta))-(\ell^*-r^*)_{T_{r}(\xi)}(\eta)+r^*_{T_{r}(\eta)}(\xi)=0$, i.e., $T_{r}^{-1}(T_{r}(\xi)\cdot T_{r}(\eta))=(\ell^*-r^*)_{T_{r}(\xi)}(\eta)-r^*_{T_{r}(\eta)}(\xi)$.
%Hence,
$T(\xi)\cdot T(\eta)=T((L^*_{ T(\xi)}-R^*_{ T(\xi)})(\eta)-R^*_{ T(\eta)}(\xi))$, that is, $T\in\OO_\A(\A^*,L^*-R^*,-R^*)$, as desired.

(3) Assume that $T\in\OO_{[\A]}(\A^*,L^*)$, then $T(\xi)T(\eta)=T(L^*_{ T(\xi)}\eta-L^*_{ T(\eta)}\xi)$.
Hence we have
\begin{eqnarray*}
&&\B(xy,z)-\B(x,y\cdot z)+\B(y,x\cdot z)\\
&=&\pair{T^{-1}(xy),z}-\pair{T^{-1}(x),y\cdot z}+\pair{T^{-1}(y),x\cdot z}\\
&=&\pair{T^{-1}(xy),z}-\pair{T^{-1}(x),L_{y}(z)}+\pair{T^{-1}(y),L_x(z)}\\
&=&\pair{T^{-1}(T(\xi)T(\eta)),z}+\pair{L^*_{T(\eta)}(\xi),z}-\pair{L^*_{T(\xi)}(\eta),z}\\
&=&\pair{T^{-1}(T(\xi)T(\eta)-T(L^*_{ T(\xi)}(\eta)-L^*_{ T(\eta)}(\xi))),z}=0.
\end{eqnarray*}
%Hence, we have $\B(xy,z)=\B(x,y\cdot z)+\B(y,x\cdot z)$.
 For the  converse statement, we suppose that $\B(xy,z)=\B(x,y\cdot z)-\B(y,x\cdot z)$. Thus
$0=\B(xy,z)-\B(x,y\cdot z)+\B(y,x\cdot z)=\pair{T^{-1}(T(\xi)T(\eta)-T(L^*_{ T(\xi)}(\eta)-L^*_{ T(\eta)}(\xi))),z}$. For the non-degeneration of the natural pairing, we see that
%$T_{r}^{-1}(T_{r}(\xi)T_{r}(\eta))+\rho_L^*(T_{r}(\eta))\xi-\rho_L^*(T_{r}(\xi))\eta=0$. Thus,
$T(\xi)T(\eta)=T(L^*_{ T(\xi)}(\eta)-L^*_{ T(\eta)}(\xi))$, i.e., $T\in\OO_{[\A]}(\A^*,L^*)$.
\end{proof}

\begin{rem}{\rm
Note that besides Malcev algebras, there are some other nonassociative algebras that contain Lie algebras as a subclass and have attracted many researchers' attention; such as Hom-Lie algebras \cite{CZ23,MS08} and $\omega$-Lie algebras \cite{CZ17,CZZZ18, Zha20}. The method of our article might be applied to a study of the CYBE on these nonassociative algebras; see \cite{Sheng14} for the study on $\OO$-operators of Hom-Lie algebras and the classical Hom-Yang-Baxter equation.
\hbo}\end{rem}

%\noindent
%{\bf Declaration of interests. }  On behalf of all authors, the corresponding author states that there is no conflict of interest.

\noindent
{\bf Acknowledgment.} This work is supported by
 National Natural Science Foundation of China (12271085).

\bibliography{}

\end{document}